\documentclass[11pt]{article}
\usepackage{amssymb}
\usepackage{amsfonts}
\usepackage{amsthm}
\usepackage{amsmath}
\usepackage{bigints}
\usepackage{authblk}
\usepackage{fullpage}
\usepackage{algorithm}
\usepackage{algorithmic}

\newtheorem{Lemma}{Lemma}

\newtheorem{Definition}{Definition}

\newtheorem{Theorem}{Theorem}

\newcommand{\gm}{\Gamma}

\usepackage[pdftex]{graphicx}
\graphicspath{{C:/Users/Sam/Desktop/PaperFigures/}}
\DeclareGraphicsExtensions{.pdf}

\begin{document}
\title{A Combinatorial Approach to Quantum Error Correcting Codes}
\author{German Luna, Samuel Reid, Bianca De Sanctis, Vlad Gheorghiu}

\date{Version of \today}
\maketitle

\begin{abstract}
Motivated from the theory of quantum error correcting codes, we investigate a combinatorial problem that involves a symmetric $n$-vertices colourable graph and a group of operations (colouring rules) on the graph: find the minimum sequence of operations that maps between two given graph colourings. We provide an explicit algorithm for computing the solution of our problem, which in turn is directly related to computing the distance (performance) of an underlying quantum error correcting code. Computing the distance of a quantum code is a highly non-trivial problem and our method may be of use in the construction of better codes.
\end{abstract}

\section{Introduction}\label{sct1}
With the recent advances in quantum information processing, protecting a physical system from environmental noise became an ultimate desideratum in the construction of a large scale quantum computer. Quantum error correcting codes, introduced about a decade ago \cite{PhysRevA.52.R2493,PhysRevA.54.1098}, play an important role in this respect and allow the protection of fragile quantum computation against undesired errors via the use of redundancy. 

An important class of quantum error correcting codes are the \emph{stabilizer codes} \cite{quantph.9705052}, which are the quantum analog of classical linear codes \cite{MacWilliamsSloane:TheoryECC}. Most known quantum error correcting codes belong to the class of stabilizer codes, which in turn are equivalent \cite{quantph.0111080} to the so-called \emph{graph codes}. The latter, as their name states, are codes built using an undirected symmetric graph on $n$ vertices. For more details about their properties see \cite{quantph.0602096, PhysRevA.78.042303, quantph.0709.1780}.

The performance of quantum codes is characterized by the \emph{distance} of the code, a positive integer that specifies how many errors the code tolerates (see e.g. \cite{NielsenChuang:QuantumComputation} for a detailed discussion). The classical analog of the quantum distance is the well-known Hamming distance \cite{MacWilliamsSloane:TheoryECC}. A vast amount of work has been dedicated to constructing quantum error correcting codes with the highest possible distance, see \cite{Grassl:codetables} for a list of up-to-date codes and bounds on the distance. 

Given a quantum error correcting code, determining its distance is highly non-trivial and in general requires numerical methods. Inspired by this, in the current article we investigate a combinatorial problem involving graphs and operations on graphs, whose solution provides the distance of a corresponding graph code. Our solution is algorithmic and, although scales exponentially\footnote{Up to our current knowledge, determining the distance of a quantum code seems to be a computationally-hard problem.} with the number of vertices of the graph, can nevertheless be of great use in the construction of quantum error correcting codes. The solution we provide is elegantly described by linear systems over finite fields with the properties of the solution space providing the distance of the code. Although the space of solutions itself can be compactly described using a suitable basis, computing the distance reduces to searching among \emph{all} possible solutions, and is this searching that makes the problem exponentially-hard. 

Even though the problem we investigate is motivated from the theory of quantum error correcting codes, we believe, given its combinatorial nature, that it may be of intrinsic interest to the discrete mathematics community, see e.g. \cite{doi:10.1137/090752237} for a related work. 

\section{Graph Labellings and Diagonal Distance}\label{sct2}
Let $G_{n} = (V,E)$ be a graph with $n$ vertices where $V$ is the vertex set of $G_{n}$ and $E$ is the edge set of $G_{n}$. A graph labelling of $G_{n}$ is a function $L: V \rightarrow \mathbb{Z}\big/ 2\mathbb{Z}$, which can be written as an element of $\left(\mathbb{Z}\big/ 2\mathbb{Z}\right)^{n}$ by using the ordering of the vertices. For example, if we have a graph with $V = \{v_{1},v_{2},v_{3},v_{4},v_{5}\}$ and a graph labelling $L = \{(v_{1},0),(v_{2},1),(v_{3},0),(v_{4},0),(v_{5},1)\}$, then we write the graph labelling simply as $(0,1,0,0,1) \in \left(\mathbb{Z}\big/ 2\mathbb{Z}\right)^{5}$ by requiring that the $i^{\text{th}}$ entry in the element of $\left(\mathbb{Z}\big/ 2\mathbb{Z}\right)^{5}$ is the label of the $i^{\text{th}}$ vertex of the graph for $1 \leq i \leq 5$. In order to simplify notation we will henceforth refer to the graph labelling $L$ as the element of $\left(\mathbb{Z}\big/ 2\mathbb{Z}\right)^{n}$ instead of as the function $L: V \rightarrow \left(\mathbb{Z}\big/ 2\mathbb{Z}\right)^{n}$. In order to define the diagonal distance of a graph $G_{n}$ we need to define two operations which indicate the possible ways to move between graph labellings. Defining these operations requires the notion of the adjacency matrix $\gm$ of a graph $G_{n}$ which is defined as the $n \times n$ matrix with $\gm_{ij} = 1$ if $(v_{i},v_{j}) \in E$ (there is an edge between $v_{i}$ and $v_{j}$) and $\gm_{ij} = 0$ if $(v_{i},v_{j}) \notin E$ (there is no edge between $v_{i}$ and $v_{j}$). Observe that $\left(\mathbb{Z}\big/ 2\mathbb{Z}\right)^{n}$ is the set of all graph labellings of a graph with $n$ vertices and so our operations our defined as follows.

\begin{Definition}
\label{dfn1}
Define two operations applied to the graph labelling $L$ of $G_{n}$ for arbitrary $v_{i} \in V$ by,
\begin{itemize}
\item $Z_{i}: \left(\mathbb{Z}\big/ 2\mathbb{Z}\right)^{n} \rightarrow \left(\mathbb{Z}\big/ 2\mathbb{Z}\right)^{n}$ defined by $L_{i} \mapsto L_{i} + 1 \mod{2}$.
\item $X_{i}: \left(\mathbb{Z}\big/ 2\mathbb{Z}\right)^{n} \rightarrow \left(\mathbb{Z}\big/ 2\mathbb{Z}\right)^{n}$ defined by $L_{j} \mapsto L_{j} + \gm_{ij} \mod{2}$, $\forall 1 \leq j \leq n$.
\end{itemize}
where $L_{i}$ denotes the $i^{\text{th}}$ entry in $L$.
\end{Definition}
\begin{figure}[h!]
\begin{center}
\includegraphics[scale=0.4]{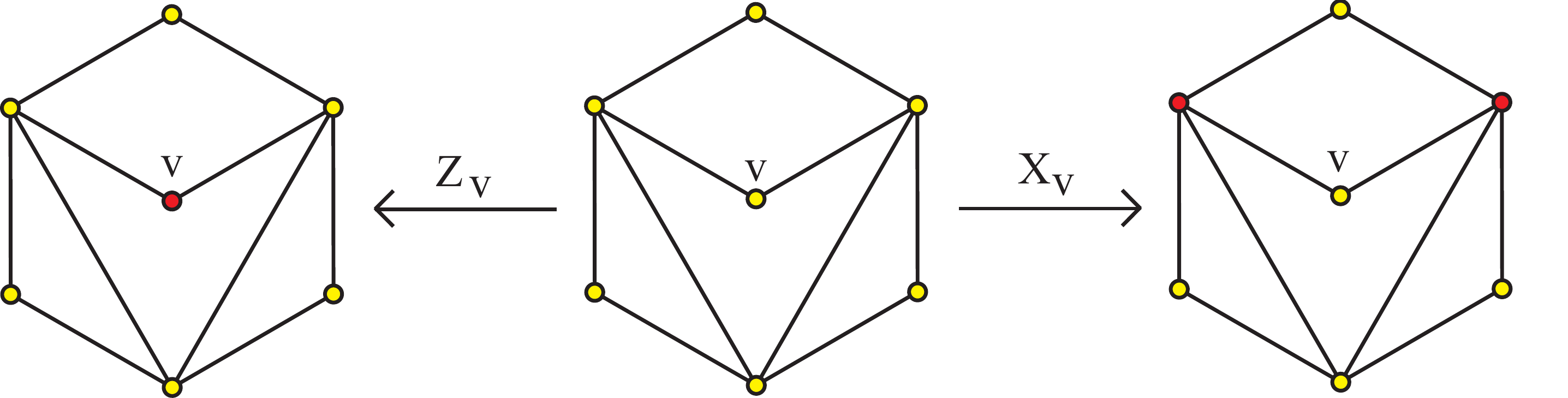}
\caption{An example of the two operations over $\left(\mathbb{Z}\big/ 2\mathbb{Z}\right)^{n}$ applied to a graph with 7 vertices.}
\end{center}
\end{figure}
We name an element in the group of operations at the $i^{\text{th}}$ vertex of $G_{n}$ by $O_{i} \in \langle Z_{i},X_{i}\rangle$; that is, $O_{i}$ is an arbitrary element of the group of operations generated by the operations $Z_{i}$ and $X_{i}$ with composition of maps as the group operation. The characteristic $\eta$-function is defined by, 
\begin{equation}\label{eqn1}
\eta(O_{i}) =
\begin{cases}
1 \;\;\; \text{if} \; O_{i} \neq I \\
0 \;\;\;\text{if} \; O_{i} = I
\end{cases}
\end{equation}
where $I: \left(\mathbb{Z}\big/ 2\mathbb{Z}\right)^{n} \rightarrow \left(\mathbb{Z}\big/ 2\mathbb{Z}\right)^{n}$ is the identity map. We are now prepared to define the diagonal distance of a graph. Intuitively, the diagonal distance of a graph $G_{n}$ is the minimum number of composed $O_{i}$'s that change the labelling $L$ of $G_{n}$ to a non-zero number of other labellings before returning to the original labelling $L$. Formally, summing the characteristic $\eta$-function over a collection of $O_{i}$'s counts the number of composed $O_{i}$'s different from the identity. By ensuring that the composition (denoted by $\prod$) of the $O_{i}$'s equals the identity map $I$, we guarantee that the graph returns to its original labelling. Notice that the diagonal distance is invariant under different numberings of the vertices, as long as the notation remains consistent throughout the usage of this algorithm.
\begin{Definition}
\label{dfn2}
The diagonal distance $\Delta(G_{n})$ of a graph $G_{n}$ is defined as,
\begin{equation}\label{eqn2}
\Delta(G_{n}) = \min\left\{\sum_{i=1}^{n} \eta(O_{i}) > 0 \; \bigg| \; \prod_{i=1}^{n}O_{i} = I, O_{i} \in \langle Z_{i},X_{i} \rangle, 1 \leq i \leq n \right\}
\end{equation}
\end{Definition}

\section{Computing Diagonal Distance Over $\mathbb{Z}\big/ 2\mathbb{Z}$}\label{sct3}
Given a graph $G_{n}$, we construct the $n \times 2n$ matrix $\Lambda$ by concatenating $I_{n}$ and $\gm$ as,
\begin{equation}\label{eqn3}
\Lambda = [I_{n} \; | \; \gm]
\end{equation}
where $I_{n}$ is the $n \times n$ identity matrix and $\gm$ is the adjacency matrix of $G_{n}$. We think of $I_{n}$ as the part of $\Lambda$ which keeps track of $Z_{i}$ applied to $G_{n}$ and $\gm$ as the part of $\Lambda$ which keeps track of $X_{i}$ applied to $G_{n}$. Let $k^{T}$ be an arbitrary nonzero $1 \times 2n$ vector with entries in $\mathbb{Z}\big/ 2\mathbb{Z}$ and write
\begin{equation}\label{eqn4}
k^{T} = [k_{1} \cdot\cdot\cdot k_{n} \; | \; k_{n+1} \cdot\cdot\cdot k_{2n}]
\end{equation}
If $\Lambda k = 0$ $(\text{mod} \; 2)$, that is, if $k$ is in the nullspace of $\Lambda$ over $\mathbb{Z}\big/ 2\mathbb{Z}$, then we claim that $k$ contains the information about the application of $O_{i}$'s to $G_{n}$ which changes the labelling $L$ of $G_{n}$ to a non-zero number of other labellings before returning to the original labelling $L$. In order to justify this claim we require the characteristic $\chi$-function defined for $1 \leq i \leq n$ by
\begin{equation}\label{eqn5}
\chi(k_{i}) =
\begin{cases}
1 \;\;\; \text{if} \; k_{i} \neq 0 \; \text{or} \; k_{i+n} \neq 0 \\
0 \;\;\;\text{if} \; k_{i} = k_{i+n} = 0
\end{cases}
\end{equation}

By relating the characteristic $\chi$-function and the characteristic $\eta$-function to each other we can relate elements $O_{i} \in \langle Z_{i},X_{i} \rangle$ to a vector $k$ that is constructed to encode the information in $O_{i}$. More precisely, we consider an arbitrary set of group operations $\{O_{i}\}_{1 \leq i \leq n}$ and for each $O_{i} = Z_{i}^{z_{i}} X_{i}^{x_{i}}$ construct the vector $k$ with the $i^{\text{th}}$ entry given by $k_{i} = z_{i} \;(mod \;2)$ and the $(i + n)^{\text{th}}$ entry given by $k_{i + n} = x_{i} \;(mod\; 2)$, where $1 \leq i \leq n$.
\begin{Lemma}\label{lma1}
Let $\{O_i\}_{i=1}^{n}$ be an arbitrary collection of $O_i\in \langle Z_i,Xi \rangle$. Then for any $O_{i}$ in this collection, there exists $z_{i},x_{i} \in \mathbb{Z}\big/ 2\mathbb{Z}$ such that $O_{i} = Z_{i}^{z_{i}} X_{i}^{x_{i}}$ and
\begin{equation}\label{eqn6}
\eta(O_{i}) = \chi(k_{i})
\end{equation}
where $k$ is the $2n \times 1$ vector with $k_{i} = z_{i}$ and $k_{i+n} = x_{i}$ for $1 \leq i \leq n$.
\end{Lemma}
\begin{proof}
Let $O_{i} = Z_{i}^{z_{i}} X_{i}^{x_{i}}$, where $z_{i},x_{i} \in \mathbb{Z}$. Note that since labellings are elements of $\mathbb{Z}\big/ 2\mathbb{Z}$, then $Z_{i}^2 = I$ and $X_{i}^2 = I$ since these maps act by adding the same quantity twice. Thus $Z_{i}^{z_{i}} X_{i}^{x_{i}} = Z_{i}^{z_{i} \mod{2}} X_{i}^{x_{i} \mod{2}}$ and so $O_{i}$ is equivalent to $Z_{i}^{z} X_{i}^{x}$ for some $x,z \in \mathbb{Z}\big/ 2\mathbb{Z}$.  Construct the vector k as stated in the lemma. If $\eta(O_{i}) = 1$, then $Z_{i}^{z} X_{i}^{x} \neq I$ and so either $z \neq 0$ or $x \neq 0$, which implies that $\chi(k_{i}) = 1$. If $\eta(O_{i}) = 0$, then $Z_{i}^{z} X_{i}^{x} = I$ and so $z = x = 0$, which implies that $\chi(k_{i}) = 0$. Therefore, $\eta(O_{i}) = \chi(k_{i})$.
\end{proof}

We now prove another lemma which restricts the possible vectors $k$, and thus the possible set of group operations $\{O_{i}\}_{1 \leq i \leq n}$, which can satisfy the definition of diagonal distance.

\begin{Lemma}\label{lma2}
Let $\{O_{i}\}_{1 \leq i \leq n}$ be a set of group operations with each $O_{i} = Z_{i}^{z_{i}} X_{i}^{x_{i}}$, where $z_{i},x_{i} \in \mathbb{Z}\big/ 2\mathbb{Z}$. Then,
\begin{equation}\label{eqn7}
\prod_{i=1}^{n} O_{i} = I \Longleftrightarrow k \in \ker(\Lambda) \big/2\mathbb{Z} 
\end{equation}
where $k$ is the $2n \times 1$ vector with $k_{i} = z_{i}$ and $k_{i+n} = x_{i}$ for $1 \leq i \leq n$.
\end{Lemma}
\begin{proof}
Write a labelling $L \in \left(\mathbb{Z}\big/ 2\mathbb{Z}\right)^{n}$ of a graph $G_{n}$ as an $n \times 1$ vector and represent $\langle Z_{i}, X_{i} \rangle$ by additive matrix operations as $Z_{i}(L) \cong L + e_{i} \mod{2}$ and $X_{i}(L) \cong L + \text{col}_{i}(\gm) \mod{2}$, where $e_{i}$ is the $n \times 1$ matrix with $1$ in the $i^{\text{th}}$ entry and $0$ elsewhere and where $\text{col}_{i}(\gm)$ is the $i^{\text{th}}$ column of $\gm$. 
($\Rightarrow$) Then for an arbitrary $O_{i} = Z_{i}^{z_{i}} X_{i}^{x_{i}}$ we have that $O_{i}$ can be represented by an action on the labelling $L$ as $O_{i} \cong L + z_{i}e_{i} + x_{i}\text{col}_{i}(\gm) \mod{2}$. Then for any set of group operations $\{O_{i}\}_{1 \leq i \leq n}$,
\begin{equation}\label{eqn8}
\prod_{i=1}^{n} O_{i} \cong L + \sum_{i=1}^{n} z_{i}e_{i} + \sum_{i=1}^{n} x_{i}\text{col}_{i}(\gm) \mod{2} = L + \Lambda k \mod{2},
\end{equation}
If $\prod_{i=1}^{n} O_{i} = I$, then $\prod_{i=1}^{n} O_{i} \cong L \mod 2$ which yields
\begin{equation}\label{eqn9}
L + \Lambda k  = L \mod{2}
\end{equation}
and therefore every set of group operations  $\{O_{i}\}_{1 \leq i \leq n}$ has a representation as a vector in the kernel of $\Lambda$. \\
($\Leftarrow$) The backward implication follows from rewriting the previous equation, since we now assume that $k$ is some arbitrary vector in the kernel. Use Lemma $1$ to construct a set of group operations corresponding to $k$. Since $k\in \text{ker}(\Lambda)/ 2 \mathbb{Z}$ by assumption, we get that
\begin{align*} L + \Lambda k = L \mod 2 \end{align*}
Then, as in the other direction, we have that $\prod_{i=1}^{n} O_i \cong L$, which implies that $\prod_{i=1}^{n} O_i=I$. \\
Therefore,
\begin{equation}\label{eqn10} 
k \in \ker(\Lambda) \big/2\mathbb{Z} \Longleftrightarrow \prod_{i=1}^{n} O_{i} = I.
\end{equation}
\end{proof}

This leads to the main theorem regarding the computation of diagonal distance.
\begin{Theorem}\label{thm1}
The diagonal distance $\Delta(G_{n})$ of a graph $G_{n}$ can be computed as
\begin{equation}\label{eqn11}
\Delta(G_{n}) = \min\left\{\sum_{i=1}^{n} \chi(k_{i}) > 0 \; \bigg| \; k \in \ker(\Lambda) \big/2\mathbb{Z} \right\}.
\end{equation}
\end{Theorem}
\begin{proof}
By Lemma \ref{lma1} and Lemma \ref{lma2} we have that,
\begin{align}\label{eqn12}
\Delta(G_{n}) &= \min\left\{\sum_{i=1}^{n} \eta(O_{i}) > 0 \; \bigg| \; \prod_{i=1}^{n}O_{i} = I, O_{i} \in \langle Z_{i},X_{i} \rangle, 1 \leq i \leq n \right\} \notag\\
&= \min\left\{\sum_{i=1}^{n} \chi(k_{i}) > 0 \; \bigg| \; k \in \ker(\Lambda) \big/2\mathbb{Z} \right\}.
\end{align}
\end{proof}
An immediate consequence of this theorem is a deterministic algorithm for computing the diagonal distance $\Delta(G_{n})$ of any graph $G_{n}$. To run this algorithm, one must first compute a basis for $\text{ker}(\Lambda)/ 2\mathbb{Z} $. Then the entire kernel must be spanned to check each vector $k$ for $\chi(k)$, and the smallest $\chi(k)$ (along with its corresponding vector, if extra information is desired) is stored. Once the algorithm terminates, the smallest of all the $\chi(k)$s will have been stored and will be equal to the diagonal distance $\Delta(G_{n})$.

\section{Generalization of Diagonal Distance to Multigraphs over $\mathbb{Z}\big/ p\mathbb{Z}$}\label{sct4}
We now generalize the diagonal distance of a graph over $\mathbb{Z}\big/ 2\mathbb{Z}$ to the diagonal distance of a multigraph (a graph that can have multiple edges between vertices) over $\mathbb{Z}\big/ p\mathbb{Z}$, where $p$ is prime. 
We define a graph labelling exactly as over $\mathbb{Z}\big/ 2\mathbb{Z}$ and note that the set of all graph labellings is now $\left(\mathbb{Z}\big/ p\mathbb{Z}\right)^{n}$. The operations $Z_{i}$ and $X_{i}$ applied to a graph labelling is defined exactly as over $\mathbb{Z}\big/ 2\mathbb{Z}$, except the adjacency matrix $\gm$ of a multigraph $G_{n}$ is defined as the $n \times n$ matrix with $\gm_{ij} = k$ if there are $k$ occurrences of $(v_{i},v_{j}) \in E$ (there are $k$ edges between $v_{i}$ and $v_{j}$) and $\gm_{ij} = 0$ if $(v_{i},v_{j}) \notin E$ (there is no edge between $v_{i}$ and $v_{j}$).
\begin{figure}[h!]
\begin{center}
\includegraphics[scale=0.4]{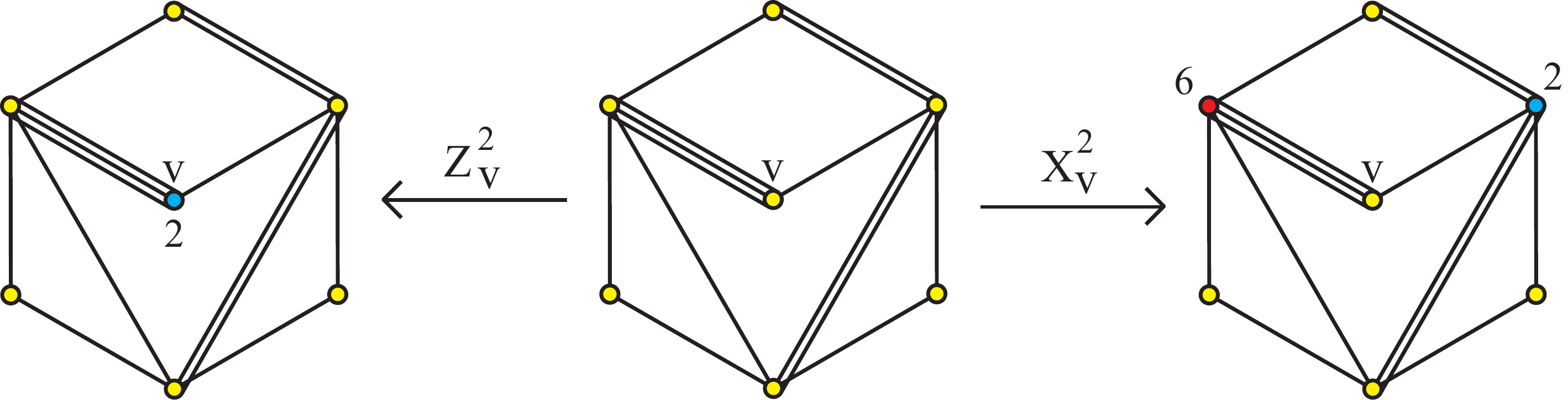}
\caption{An example of the two operations $Z_{i}^2$ and $X_{i}^2$ applied to a multigraph with 7 vertices.}
\end{center}
\end{figure}
We can then define the group of operations $O_{i} = \langle Z_{i},X_{i} \rangle$, with our operations defined over $\mathbb{Z}\big/ p\mathbb{Z}$ in order to obtain the straightforward generalization of Theorem \ref{thm1} to $\mathbb{Z}\big/ p\mathbb{Z}$ by writing mod $p$ everywhere that mod $2$ is written in the proofs of Lemma \ref{lma1}, Lemma \ref{lma2}, and Theorem \ref{thm1}.

\begin{Theorem}\label{thm2}
The diagonal distance $\Delta(G_{n})$ of a graph $G_{n}$ can be computed as,
\begin{equation}\label{eqn13}
\Delta(G_{n}) = \min\left\{\sum_{i=1}^{n} \chi(k_{i}) > 0 \; \bigg| \; k \in \ker(\Lambda) \big/p\mathbb{Z} \right\}.
\end{equation}
\end{Theorem}

\section{Minimal Diagonal Distance of Codes Over $\mathbb{Z}\big/ p\mathbb{Z}$}\label{sct5}
Thus far we have considered the diagonal distance $\Delta(G_{n})$ over $\mathbb{Z}\big/ 2\mathbb{Z}$ of a graph $G_{n}$ and, more generally, the diagonal distance $\Delta(G_{n})$ over $\mathbb{Z}\big/ p\mathbb{Z}$ of a multigraph $G_{n}$. In either of these cases, we built into the definition of diagonal distance that $\prod_{i=1}^{n} O_{i} = I$ where $I: \left(\mathbb{Z}\big/ p\mathbb{Z}\right)^{n} \rightarrow \left(\mathbb{Z}\big/ p\mathbb{Z}\right)^{n}$ is the identity map. Yet, we can extend this notion to requiring that $\prod_{i=1}^{n} O_{i} = L$, where $L: \left(\mathbb{Z}\big/ p\mathbb{Z}\right)^{n} \rightarrow \left(\mathbb{Z}\big/ p\mathbb{Z}\right)^{n}$ is some specific map, e.g. the map sending the configuration $\vec{0}$ to $\vec{1}$.

Let $\{C_{1},...,C_{k}\}$ be a set of graph labellings (i.e. codewords), where $C_{1} = \vec{0}$. We then want to determine the diagonal distance between any two graph labellings $C_{r}$ and $C_{s}$ in this set, the distance of the code, which can be defined as follows.

\begin{Definition}
\label{dfn3}
For a multigraph $G_{n}$ and a set of graph labellings $\{C_{1},C_{2},...,C_{k}\}$, the diagonal distance $\Delta_{rs}(G_{n})$ between two graph labellings $C_{r}$ and $C_{s}$ is defined as,
\begin{equation}\label{eqn14}
\Delta_{rs}(G_{n}) = \min\left\{\sum_{i=1}^{n} \eta(O_{i})>0 \; \bigg| \; \prod_{i=1}^{n}O_{i} = L_{rs}, O_{i} \in \langle Z_{i},X_{i} \rangle, 1 \leq i \leq n \right\}
\end{equation}
where $G_{n}$ is assumed to be originally labelled by $C_{r}$ and $L_{rs}$ is the map sending the graph labelling $C_{r}$ to the labelling $C_{s}$.
\end{Definition}

Similarly to computing $\Delta(G_{n})$, we can minimize $\Delta_{rs}(G_{n})$ by minimizing the summation of the characteristic $\chi$-function of vector $k^\prime$ such that $\Lambda k^\prime = C_{r}-C_{s} \mod{p}$.  It is a straightforward generalization of Lemma \ref{lma1} and Lemma \ref{lma2} in order to obtain the analogue of Theorem \ref{thm2} for distance of a code.
\begin{Theorem}\label{thm3}
The generalized diagonal distance $\Delta_{rs}(G_{n})$ of a graph $G_{n}$ can be computed as,
\begin{equation}\label{eqn15}
\Delta_{rs}(G_{n}) = \min\left\{\sum_{i=1}^{n} \chi(k_{i}^\prime) > 0 \; \bigg| \; \Lambda k^\prime = C_{r}-C_{s} \mod{p} \right\}
\end{equation}
where $G_{n}$ is originally labelled by $C_{r}$.
\end{Theorem}

Note that the distance $\delta$ of a code $\{C_{1},C_{2},...,C_{k}\}$ is simply min$\{ \Delta_{rs}(G_{n}) \; \mid \;  r,s \in \{1,2,...k\} \}$

\section{Computational Example}\label{sct6}
We run the main algorithm on the $5$-cycle graph in $\mathbb{Z}/2\mathbb{Z}$. It is well-known that the diagonal distance of the $5$-cycle is $3$, for example by the set of operations shown in Figure 3. This can be found directly using our algorithm.
\begin{figure}[h!]
\begin{center}
\includegraphics[scale=0.4]{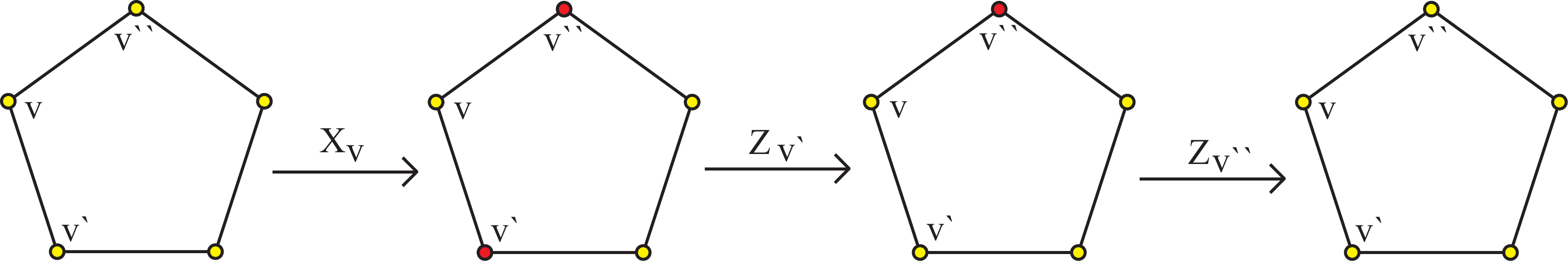}
\caption{An example of three operations that illustrate the diagonal distance of the $5$-cycle.}
\end{center}
\end{figure}
To compute the diagonal distance, we must first set up the matrix $\Lambda = [I_n \; | \; \Gamma]$. For the $5$-cycle, this is a $5 \times 10$ matrix which has many possibilities depending on how the vertices are numbered. However, diagonal distance is invariant under different numberings and so we will just pick one.
\begin{align*} \Lambda = 
\begin{bmatrix}
1 & 0 & 0 & 0 & 0  & 0 & 1 & 0 & 0 & 1\\
0 & 1 & 0 & 0 & 0 & 1 & 0 & 1 & 0 & 0 \\
0 & 0 & 1 & 0 & 0 & 0 & 1 & 0 & 1 & 0\\
0 & 0 & 0 & 1 & 0 & 0 & 0 & 1 & 0 & 1\\
0 & 0 & 0 & 0 & 1 & 1 & 0 & 0 & 1 & 0
\end{bmatrix} \end{align*}
Next we must compute the kernel of this matrix. Some computation gives that the kernel has the following basis vectors.
\begin{align*} k_1 &= (0,1,0,0,1,1,0,0,0,0)^{T} \\
k_2&= (1,0,1,0,0,0,1,0,0,0)^{T} \\
k_3&= (0,1,0,1,0,0,0,1,0,0)^{T}\\
k_4&= (0,0,1,0,1,0,0,0,1,0)^{T}\\
k_5&= (1,0,0,1,0,0,0,0,0,1)^{T}
\end{align*}
To find the vector in the kernel with the least amount of $1$'s in it, we must search through all the linear combinations of these basis vectors. It is easy enough to write a program to do this. For this example, it actually turns out that the linear combinations with the least amount of $1$'s in them are exactly the basis vectors themselves. Notice that all of these vectors have $3$ $1$'s in them. Therefore, the diagonal distance of the $5$-cycle graph is $3$.

\section*{Acknowledgments}
V. Gheorghiu acknowledges support from the Natural Sciences and Engineering Research Council (NSERC) of Canada and from the Pacific Institute for the Mathematical Sciences (PIMS).


\noindent
German Luna
\newline
Department of Mathematics and Statistics, University of Calgary, Calgary AB T2N 1N4, Canada
\newline
{\sf E-mail: galunapa@ucalgary.ca}
\newline
\newline
Samuel Reid
\newline
Department of Mathematics and Statistics, University of Calgary, Calgary AB T2N 1N4, Canada
\newline
{\sf E-mail: smrei@ucalgary.ca} 
\newline
\newline
Bianca De Sanctis
\newline
Department of Mathematics and Statistics, University of Calgary, Calgary AB T2N 1N4, Canada
\newline
{\sf E-mail: bddesanc@ucalgary.ca}
\newline
\newline
Vlad Gheorghiu
\newline
Department of Mathematics and Statistics, University of Calgary, Calgary AB T2N 1N4, Canada
\newline
Institute for Quantum Science and Technology, Calgary AB T2N 1N4, Canada
\newline
{\sf E-mail: vgheorgh@ucalgary.ca}

\end{document}